\documentclass[a4paper,12pt,leqno]{amsart}
\usepackage[active]{srcltx}
\usepackage{verbatim}
\usepackage{epsfig,graphicx,color,mathrsfs}
\usepackage{graphicx}
\usepackage{amssymb}
\usepackage[english]{babel}
\usepackage[T1]{fontenc}
\usepackage[latin9]{inputenc}
\usepackage{amsmath}
\usepackage{amsthm}
\usepackage{amssymb}
\usepackage{xcolor}

\makeatletter
\theoremstyle{plain}
\newtheorem{thm}{\protect\theoremname}
\theoremstyle{plain}
\newtheorem{defn}[thm]{\protect\definitionname}
\theoremstyle{plain}
\newtheorem{prop}[thm]{\protect\propositionname}
\theoremstyle{plain}
\newtheorem{lem}[thm]{\protect\lemmaname}
\theoremstyle{plain}

\theoremstyle{plain}

\theoremstyle{plain}
\newtheorem{cor}[thm]{\protect\corollaryname}

\makeatother

\usepackage{babel}
\providecommand{\lemmaname}{Lemma}
\providecommand{\definitionname}{Definition}
\providecommand{\propositionname}{Proposition}
\providecommand{\theoremname}{Theorem}
\providecommand{\Remarkname}{Remark}
\providecommand{\conjecturename}{Conjecture}
\providecommand{\corollaryname}{Corollary}

\newcommand{\Rn}{\ensuremath  \mathbb{R}^N_+}

\thanks{$^*$ corresponding author}

\title{Classification of  solutions to $-\Delta u={u^{-\gamma}}$ in the half-space}

\email{montoro@mat.unical.it., muglia@mat.unical.it, sciunzi@mat.unical.it}

\address{Department of Mathematics and Computer Science, UNICAL, Rende (CS), Italy}
\author[L.\ Montoro]{Luigi Montoro}
\author[L.\ Muglia]{Luigi Muglia}
\author[B.\ Sciunzi]{Berardino Sciunzi $^*$}

\begin{document}

\begin{abstract}
We provide a classification result for positive solutions to $-\Delta u=\frac{1}{u^\gamma}$ in the half space, under zero Dirichlet boundary condition. 
\end{abstract}
%\MSC[2020]{35J75, 35A02, 35B09}

\maketitle

\section{Introduction}
We deal with the classification of positive solutions to the singular problem 
\begin{equation}\tag{$\mathcal P_{\gamma}$}\label{MP}
\begin{cases}
\displaystyle -\Delta u=\frac{1}{u^\gamma} & \text{in} \,\, \mathbb{R}^N_+\\
u>0 & \text{in} \,\,  \mathbb{R}^N_+\\
u=0 & \text{on} \,\, \partial \mathbb{R}^N_+,
\end{cases}
\end{equation}
where $N\geq 1$, $\gamma >1$, $x\in \mathbb{R}^N_+$ is represented by $x=(x',x_N)$, $x'\in \mathbb{R}^{N-1}$ and $\mathbb{R}^N_+:=\{x\in \mathbb{R}^N:x_N>0\}$.
This is a captivating problem itself but it also arises in the study of limiting scaling arguments at the boundary,  in bounded domains, for solutions to 
\[
-\Delta u=\frac{1}{u^\gamma}+f(u)
\]
since the term ${1}/{u^\gamma}$ is the leading part near zero, see e.g. \cite{CaEsSc, EsSc}. 
Although the expert reader may guess that the decreasing nature of the nonlinearity is favourable for the application of maximum and comparison principles, we stress that here, solutions are not in the right Sobolev space in order to do that (in particular for $\gamma>1$). This causes a deep and challenging issue that we analyze  introducing some new ideas.\\

\noindent Taking into account the nature of the problem, it follows that, the natural assumption that we shall adopt in all the paper is
 $$u\in C^2(\mathbb{R}^N_+)\cap C^0(\overline{\mathbb{R}^N_+})\,.$$ 
 Note that the continuity up to the boundary of the solutions can be proved as in \cite{CraRaTa}. Therefore the equation is understood in the classic meaning in the interior of the domain, or in the variational meaning as in the following:
 \begin{equation}\label{w-sol}
 \int_{\Rn}\nabla u \cdot\nabla \varphi=\int_{\Rn}\frac{1}{u^\gamma}\varphi \qquad \forall \varphi\in C^1_c(\mathbb{R}_+^N).
\end{equation}

We will classify all the locally bounded solutions according to the following hypothesis\ \\

 \noindent {\bf(hp)} There exists $\bar\lambda>0$  such that  $u$ is bounded on the set $\Sigma_{\bar\lambda}$. We set $\theta\in \mathbb R$ such that 
\[
\theta:=\sup_{\Sigma_{\bar\lambda}}u(x).%\theta:=\sup_{0\leq x_N\leq \bar \lambda }u.
\]
where the strip  $\Sigma_{\bar\lambda}$  is defined in Section  \ref{sezpde}.
\noindent Our main result is the following
\begin{thm}\label{mainthm}
 Let $u$ be a solution of \eqref{MP} fulfilling ${\bf (hp)}$. Then 
 \[u(x)=u(x', x_N)=u(x_N).\] 
 Consequently  either
$$
  u(t)=\frac{(\gamma+1)^\frac{2}{\gamma+1}}{(2\gamma-2)^\frac{1}{\gamma+1}}t^\frac{2}{\gamma+1}
 $$
or
$$
 u(t)=\lambda^{-\frac{2}{\gamma+1}}v(\lambda t) \quad \lambda>0,
$$
being $v(t)\in C^2(\mathbb{R}_+)\cap C(\overline{\mathbb{R}_+})$ the unique solution to
\begin{eqnarray}
\begin{cases}
\displaystyle - v''=\frac{1}{v^\gamma} & t>0\\
v(t)>0 &t>0 \\
v(0)=0\quad  \displaystyle \lim_{t\to+\infty}v'(t)=1.
\end{cases}
\end{eqnarray}
\end{thm}

The starting crucial issue in our proof is the study of the accurate asymptotic behavior  of the solutions up to the boundary as well as at infinity. Actually we shall show that every solution has at most linear growth far from the boundary, which is a sharp estimate. This analysis will allow us to exploit a celebrated result of Berestycki,  Caffarelli and Nirenberg \cite{BeCaNi} to deduce that the solutions exhibit 1-D symmetry. Then, taking into account the non standard nature of the equation arising from the singular term, we will carry out a ODE analysis to complete our proof.
%
%A consequently ODE analysis, carried out taking into account the non standard nature of the equation arising from the singular term,  will complete our proof.

The paper is organized as follows: in Section \ref{sezpde} we provide the proofs of the asymptotic analysis and exploit it to prove the 1-D result.
 In Section \ref{sezode} we carry out the ODE analysis. We conclude in Section \ref{Sec4} with the proof of our main result.

\section{Asymptotic analysis and 1-D symmetry}\label{sezpde}

\noindent In all the paper we shall use the notation given in the following
\begin{defn}
	Given  $0<a<b$, we define the strip $\Sigma_{(a,b)}$ as the set given by
	\begin{equation}\label{eq:strip} 
		\Sigma_{(a,b)}:=\{x\in \mathbb R^N_+\,:\, a < x_N < b\}.
	\end{equation}
	We also  set $\Sigma_{(0,b)}:=\Sigma_{b}$.
\end{defn}
In all this  section we will use some ODEs arguments that are actually contained in the more general analysis of Section \ref{sezode}.
\noindent We start proving
\begin{lem}\label{thm:spr}
Under the assumption ${\bf (hp)}$, it follows that  
\[u(x)\leq C x_N^{\frac{2}{\gamma+1}}\qquad  \text{in} \,\, \Sigma_{\bar \lambda},\]
with $C=C(\gamma, \theta)$ a positive constant.  
\end{lem}
\begin{proof}
Let us consider the $1$-D solution $w(x_N)$ of 
\begin{equation*}
\begin{cases}
\displaystyle -w''=\frac{1}{w^\gamma}& \text{in}\,\, \mathbb R^+\\
w(0)=0\\
w>0& \text{in}\,\, \mathbb R^+,
\end{cases}
\end{equation*} 
given in \eqref{eq:esam3}. Note that $w_\beta=\beta w$ solves $-\Delta w_\beta=\beta^{\gamma+1}/w_\beta^\gamma$. Therefore, for $\beta >1$, we have 
\begin{equation}\label{eq:wbeta}
\begin{cases}
\displaystyle -\Delta w_\beta> \frac{1}{w_\beta} & \text{in}\,\, \mathbb R_N^+\\
w_\beta>0 & \text{in} \,\, \mathbb{R}^N_+\\
w_\beta=0 & \text{on }\,\, \partial \mathbb{R}^N_+.
\end{cases}
\end{equation}
Now, since 
\begin{equation}\label{eq:pen1}
w_\beta =\beta w(\bar \lambda)\quad \text{on}\,\, \{x\in \mathbb R^N_+\,:\,  x_N= \bar \lambda\},\end{equation}
we take $\beta$ large so that, for $\beta\geq \theta/w(\bar \lambda)$ we deduce
\[w_\beta\geq\beta w(\bar \lambda)\geq \theta.\]
Consequently $u,w_\beta$ are well ordered on the boundary of the strip $\Sigma_{\bar\lambda}$ (see \eqref{eq:strip}), namely 
\begin{equation}\label{eq:esam2}
u\leq w_\beta\quad \text{on}
\,\, \partial \Sigma_{\bar\lambda}.
\end{equation}
In order to prove  a comparison principle, we have to take into account that  we are working  in the unbounded domain  $\mathbb R^N_+$ and both $u,w_\beta$ lose regularity at the boundary of the half-space $\mathbb R^N_+$.  For this reason we start  defining 
$\phi_R(x'):\mathbb R^{N-1}\rightarrow \mathbb R$ such that
\begin{equation}\label{eq:esam1}
\begin{cases}
\phi_R(x')=1 & \text{in}\,\, B'_R(0)\\
\phi_R(x')=0 & \text{in}\,\, \mathbb R^{N-1}\setminus B'_{2R}(0)\\
\displaystyle |\nabla\phi_R(x') |\leq \frac{C}{R}  & \text{in}\,\, \mathbb R^{N-1},
\end{cases}
\end{equation}
where we recall that a point $x\in \mathbb \Rn$ is denoted by $x=(x',x_N)$ with $x'\in \mathbb R^{N-1}$ and where $$B'_R(0):=\{x'\in \mathbb R^{N-1}\,:\, |x'|<R\}.$$
Moreover, let us define the translated function (indeed still a supersolution to~\eqref{eq:wbeta})
\[w_{\beta, \varepsilon}=w_\beta (x_N+\varepsilon)\]
and let $\varphi_R$  defined as 
\[\varphi_R=(u-w_{\beta, \varepsilon})^+\phi_R^2.\] 
One can check, using a suitable argument based on the continuity of $u$ and $w_{\beta, \varepsilon}$,  that  $\varphi_R$   is indeed a suitable function test to both problems \eqref{MP} and \eqref{eq:wbeta}.
Let us also define the cylinder $${C}(R):=\left\{ \Sigma_{\bar\lambda}\cap \overline{\{B'_R(0)\times \mathbb{R}\}} \right\}.$$
Then using $\varphi_R$ in the weak formulations satisfied by $u$ and by $w_\varepsilon$,  we obtain 
\begin{eqnarray*}
&&\int_{C(2R)}(\nabla u,\nabla(u-w_{\beta, \varepsilon})^+)\phi_R^2\, dx+2 \int_{C(2R)}(\nabla u,\nabla\phi_R)\phi_R (u-w_{\beta, \varepsilon})^+\, dx\\
&&=\int_{C(2R)}\frac{1}{u^\gamma}(u-w_{\beta, \varepsilon})^+\phi_R^2\, dx
\end{eqnarray*}
and
\begin{eqnarray*}
&&\int_{C(2R)}(\nabla w_{\beta, \varepsilon},\nabla(u-w_{\beta, \varepsilon})^+)\phi_R^2\, dx+2 \int_{C(2R)}(\nabla w_{\beta, \varepsilon},\nabla\phi_R)\phi_R (u-w_{\beta, \varepsilon})^+\, dx\\
&&\geq \int_{C(2R)}\frac{1}{w_{\beta, \varepsilon}^\gamma}(u-w_{\beta, \varepsilon})^+\phi_R^2\, dx.
\end{eqnarray*}
Subtracting the last inequalities we obtain
\begin{eqnarray}\label{eq:mad 2}
&&\int_{C(2R)}|\nabla (u-w_{\beta, \varepsilon})^+|^2\phi_R^2\, dx \\\nonumber
&\leq& 2\int_{C(2R)}|\nabla (u-w_{\beta, \varepsilon})^+||\nabla \phi_R|\phi_R(u-w_{\beta, \varepsilon})^+\, dx\\\nonumber &&+\int_{C(2R)}\left(\frac{1}{u^\gamma}-\frac{1}{w_{\beta, \varepsilon}^\gamma}\right)(u-w_{\beta, \varepsilon})^+\phi_R^2\, dx \\\nonumber
&\leq& 
2\int_{C(2R)}|\nabla (u-w_{\beta, \varepsilon})^+||\nabla \phi_R|\phi_R(u-w_{\beta, \varepsilon})^+\, dx\\\nonumber &&+\int_{C(2R)}\left(\frac{1}{u^\gamma}-\frac{1}{w_{\beta, \varepsilon}^\gamma}\right)\frac{[(u-w_{\beta, \varepsilon})^+]^2}{(u-w_{\beta, \varepsilon})}\phi_R^2\, dx. 
\end{eqnarray}
We  observe  that there exists a positive constant  $\eta=\eta(\gamma,\theta)$ such that
\[\left(\frac{1}{u^\gamma}-\frac{1}{w_{\beta, \varepsilon}^\gamma}\right)\frac{1}{u-w_{\beta, \varepsilon}}\leq -\eta<0,\]
since $u$ is bounded on $\Sigma_{\bar\lambda}$ by ${\bf (hp)}$. Moreover using   Young inequality and \eqref{eq:esam1}, we  also deduce that 
\begin{eqnarray}\label{eq:mad 3}
&&\int_{C(2R)}|\nabla (u-w_{\beta, \varepsilon})^+||\nabla \phi_R|\phi_R(u-w_{\beta, \varepsilon})^+\, dx\\\nonumber
&&\leq \delta\int_{C(2R)}|\nabla (u-w_{\beta, \varepsilon})^+|^2\phi_R^2\, dx+\frac{C(\delta)}{R^2}\int_{C(2R)}[(u-w_{\beta, \varepsilon})^+]^2\, dx.
\end{eqnarray}
 Therefore, using \eqref{eq:mad 3} in \eqref{eq:mad 2}, we get 
 \begin{eqnarray}\nonumber
&&\int_{C(2R)}|\nabla (u-w_{\beta, \varepsilon})^+|^2\phi_R^2\, dx \\\nonumber
&\leq& 2\delta\int_{C(2R)}|\nabla (u-w_{\beta, \varepsilon})^+|^2\phi_R^2\, dx+\frac{2C(\delta)}{R^2}\int_{C(2R)}[(u-w_{\beta, \varepsilon})^+]^2\, dx\\\nonumber
&& -2\eta \int_{C(2R)}[(u-w_{\beta, \varepsilon})^+]^2\, dx. 
\end{eqnarray}
For $\delta$ small fixed we deduce that 
 \begin{eqnarray}
\nonumber &&(1-2\delta) \int_{C(2R)}|\nabla (u-w_{\beta, \varepsilon})^+|^2\phi_R^2\, dx \leq
\frac{2C(\delta)}{R^2}\int_{C(2R)}[(u-w_{\beta, \varepsilon})^+]^2\, dx\\
&&-2\eta \int_{C(2R)}[(u-w_{\beta, \varepsilon})^+]^2\, dx. 
\end{eqnarray}
For $R$ large we have that ${C(\delta)}/{R^2}<\eta$  and therefore 
\[\int_{C(2R)}|\nabla (u-w_{\beta, \varepsilon})^+|^2\phi_R^2\, dx \leq 0.
\]
By Fatou's Lemma for $R\rightarrow +\infty$,  we  obtain
\[\int_{\Sigma_{\bar \lambda}}|\nabla (u-w_{\beta, \varepsilon})^+|^2\, dx \leq 0.
\]
 Exploiting  \eqref{eq:esam2} we deduce that actually
 \[u\leq w_{\beta, \varepsilon},\quad \text{for all}\,\, \varepsilon >0.\] Finally, by continuity, we have that $u\leq w_{\beta}$. Recalling \eqref{eq:pen1} and  \eqref{eq:pen2} we get thesis. 
\end{proof}
%\begin{eqnarray}\label{prob:madr}
%\begin{cases}
%\displaystyle -\Delta u=\frac{1}{u^\gamma}, & in \mathbb R^+_N\\
%u>0, & in \mathbb R^+_N\\
%u=0, & on \mathbb R^+_N.
%\end{cases}
%\end{eqnarray}
\noindent Without assuming any a priori assumption, we prove the following 
\begin{lem}\label{thm:Lazer}
There exists a constant $C=C(\gamma)$ such that 
\[u \geq C x_N^{\frac{2}{\gamma+1}}\quad \text{in}\,\, \mathbb R^N_+.\]
\end{lem}
\begin{proof}
Let us  consider  the first eigenfunction $\varphi_1 \in C^2(\overline{B_1(0)})$ solution to 
\begin{equation}\label{eq:autv}
\begin{cases}
-\Delta \varphi_1=\lambda_1\varphi_1 & \text{in}\,\, B_1(0)\\
\varphi_1>0 & \text{in}\,\, B_1(0)\\
\varphi_1 =0 & \text{on}\,\, \partial B_1(0).
\end{cases}
\end{equation}
Setting \[w=C\varphi_1^{\frac{2}{\gamma+1}},\]   with $C>0$ to be chosen, by a straightforward computations 
\begin{equation*}
\Delta w= \frac{2C(1-\gamma)}{(1+\gamma)^2}\varphi_1^{-\frac{2\gamma}{1+\gamma}}|\nabla \varphi_1|^2+\frac{2C}{1+\gamma}\varphi_1^{\frac{1-\gamma}{1+\gamma}}\Delta \varphi_1.
\end{equation*}
Using \eqref{eq:autv}, we obtain  
\begin{eqnarray*}
-\Delta w&=&\frac{2C(\gamma-1)}{(1+\gamma)^2}\varphi_1^{-\frac{2\gamma}{1+\gamma}}|\nabla \varphi_1|^2+\frac{2C\lambda_1}{1+\gamma}\varphi_1^{\frac{2}{1+\gamma}}\\\nonumber
&=& \frac{1}{C^\gamma\varphi_1^\frac{2\gamma}{\gamma+1}}\left(\frac{2C^{\gamma+1}(\gamma-1)}{(\gamma+1)^2}|\nabla \varphi_1|^2+ \frac{2\lambda_1C^{\gamma +1}}{\gamma+1}\varphi_1^2\right)\\\nonumber
&:=&\frac{\alpha(x)}{w^\gamma}\quad \text{in}\,\, B_1(0).
\end{eqnarray*}
For $C=C(\gamma)$ small enough, we get $\alpha(x)<1$ and therefore $w$ is a subsolution to $-\Delta w= w^{-\gamma}$ in $B_1(0)$. 

Let now $x_0=(x_0', x_{0,N})\in \mathbb R^N_+$ and set
\[w_{x_0, R}=R^{\frac{2}{\gamma+1}}w\left (\frac{x-x_0}{R}\right)\quad\text{in}\,\, B_R(x_0),\]\
where $R= x_{0,N}$. We have
\begin{eqnarray}
-\Delta w_{x_0, R}&=&-R^{-\frac{2\gamma}{\gamma+1}}\Delta w\left (\frac{x-x_0}{R}\right)\\\nonumber
&\leq& \frac{1}{R^{\frac{2\gamma}{\gamma+1}}  w^{\gamma}\left (\frac{x-x_0}{R}\right)}=\frac{1}{w_{x_0,R}^\gamma}\quad\text{in}\,\, B_R(x_0).
\end{eqnarray}
Let $u$ be a solution to \eqref{MP}; we observe that 
\begin{equation}\label{eq:mang1}
\begin{cases}
-\Delta u= \frac{1}{u^\gamma}& \text{in}\,\, B_R(x_0)\\
\\
-\Delta w_{x_0, R}\leq  \frac{1}{w_{x_0, R}^\gamma}& \text{in}\,\, B_R(x_0).
\end{cases}
\end{equation}
For  $\varepsilon>0$, we can use\[(w_{x_0, R}-u -\varepsilon)^+\]
as a test function in \eqref{eq:mang1} obtaining that 

%Arguing   as  in the proof of Theorem \ref{thm:spr}, we can use for $\varepsilon>0$
%\[(w_{x_0, R}-u -\varepsilon)^+\]
%as a function test in \eqref{eq:mang1} obtaining that 
\[w_{x_0, R}\leq u +\varepsilon, \quad \text{for all}\,\, \varepsilon >0.\]
Then $u\geq w_{x_0, R}$ in $B_R(x_0)$, hence 
\[u(x_0)=R^{\frac{2}{\gamma+1}}w(0)=C(R\varphi_1(0))^{\frac{2}{\gamma+1}}.\] Recalling that $R= x_{0,N}$, since $x_0$ is arbitrary,  we obtain the thesis. 
\end{proof}
\begin{prop}\label{thm:lineare}
Under the assumption {\bf (hp)}, there exists a positive constant $C=C(\gamma, \theta, N)$ such that 
\[u(x)\leq C x_N\]
in the set $\mathbb R^N_+\setminus \Sigma_{\bar \lambda}$.
\end{prop}
\begin{proof}
In what follows, without loss of generality, from $\bf(hp)$,  using the natural scaling for the problem \eqref{MP}\begin{equation}\label{eq:ciserveallafine}
u_{\frac{\bar \lambda}{2}}(x)=\left(\frac{\bar \lambda}{2}\right)^{-\frac{2}{\gamma+1}}u\left(\frac{\bar \lambda}{2}x\right),
\end{equation}
we may assume that our solution $u$ is indeed bounded in the strip $\Sigma_{(0,2)}$.\ \\   
Let $x_0\in \mathbb R^+_N$, $x_0=(x_0';x_{0,N})$,  with $x_{0,N}>2$ and let R>0, such that  
\begin{equation}\label{eq:xnscaling}
x_{0,N}=4R.\end{equation}
%We have that 
%
%\[-\Delta u =\frac{1}{u_\gamma}\quad\text{in}\,\, B_{4R}(x_0).\]
Let $u_R(x)=u(x_0+R(x-x_0))$; then 
\begin{equation}\label{eq:4R}
-\Delta u_R=R^2\frac{1}{u_R^{\gamma+1}}u_R\quad \text{in} \,\, B_4(x_0),
\end{equation}
$u_R>0$ in $B_4(x_0)$.
Since in Lemma \ref{thm:Lazer} we showed that 
\[u\geq C x_N^{\frac{2}{\gamma+1}},\]
in the whole $\mathbb R^N_+$, we infer that 
\[u^{\gamma+1}_R(x)\geq 4C^{\gamma+1} R^2\quad\text{in}\,\, B_2(x_0),\]
where $C$ is the positive constant given in Lemma \ref{thm:Lazer}. Therefore 
\[c(x):=\frac{R^2}{u_R^{\gamma+1}}\leq \frac{R^2}{4C^{\gamma+1}R^2}\quad  \text{in}\,\, B_2(x_0).\]
We point out that, from  the arbitrariness  of $x_0$, we deduce that 
\[c(x) \leq C(\gamma)\quad  \text{in}\,\,\Sigma_{(\frac52,\frac{11}{2})}.\]
Consequently from \eqref{eq:4R}, we deduce that 
\[-\Delta u_R= c(x) u_R\quad \text{in}  \,\,B_2(x_0).\]
By Harnack inequality \cite[Theorem 8.20]{GT} we have that 
\begin{equation}\label{eq:mad12-1}
\sup_{B_1(x_0)} u_R\leq C_H \inf _{B_1(x_0)} u_R,\end{equation}
where $C_H=C_H(\gamma,N)$.
Now let us consider, for $N\geq 3$, the fundamental solution of the Laplace operator. So let us define 
\[v_{c,k}=c\left (\frac{1}{|x-x_0|^{N-2}}+k\right)\]
that fulfills 
\[\Delta v_{c,k}=0\quad \text{in}\,\, \mathbb{R}^N\setminus\{x_0\},\]
for all $c,k\in\mathbb R$. Exploiting \eqref{eq:mad12-1} with $u_0=u(x_0)$, we infer that 
\[u_0\leq\sup_{B_R(x_0)}u=\sup_{B_1(x_0)} u_R  \leq C_H  \inf _{B_1(x_0)} u_R=C_H\inf _{B_R(x_0)} u\leq C_H u(x), \]
hence
\[u(x)\geq C^{-1}_Hu_0\quad \text{on}\,\, \partial B_R(x_0).\] 
We new choose $c$ amd $k$ such that 
\begin{equation}\label{eq:mad12-2}
\begin{cases}
v_{c,k}=C^{-1}_Hu_0 & \text{on}\,\, \partial B_{2R}(x_0) \\
v_{c,k}=0 & \text{on}\,\, \partial B_{4R}(x_0). 
\end{cases}
\end{equation}
Direct computation shows that the system \eqref{eq:mad12-2} holds for 
\begin{equation}\label{eq:mad12-2_1}
c=\frac{C^{-1}_Hu_0(4R)^{N-2}}{2^{N-2}-1}:=\tilde c_N u_0 R^{N-2}\quad \text{and}\quad k=-\frac{1}{(4R)^{N-2}},\end{equation}
with  $\tilde c_N=C^{-1}_H4^{N-2}/({2^{N-2}-1})$.
Summarizing we have that  
\begin{equation}\label{eq:probfondu}
\begin{cases}
\displaystyle -\Delta u=\frac{1}{u^\gamma}\geq 0 & \text{in}\,\, B_{4R}(x_0)\setminus B_{2R}(x_0)\\
-\Delta v_{c,k}=0 & \text{in}\,\, B_{4R}(x_0)\setminus B_{2R}(x_0)\\
u,v_{c,k}>0 & \text{in}\,\, B_{4R}(x_0)\setminus B_{2R}(x_0).
\end{cases}
\end{equation}
Using $(v_{c,k}-u-\varepsilon)^+$, for $\varepsilon >0$,  as test function in \eqref{eq:probfondu} (see also \eqref{eq:mad12-2}), we get 
\[\int_{B_{4R}(x_0)\setminus B_{2R}(x_0)}|\nabla (v_{c,k}-u-\varepsilon)^+|^2\, dx\leq 0,\]
namely $v_{c,k}\leq u+\varepsilon,$ for all $\varepsilon >0$.
Therefore \begin{equation}\label{eq:mad12-3}
u(x)\geq v_{c,k}\quad \text{in}\,\, B_{4R}(x_0)\setminus B_{2R}(x_0).
\end{equation}
Therefore
\begin{eqnarray}\nonumber
&&u(x_0',1)\geq v_{c,k}(x_0',1)\\\nonumber
&&=c\left (\frac{1}{|(x_0',1)-(x_0', x_{0,N})|^{N-2}} +k\right)=c\left (\frac{1}{|1-4R|^{N-2}} +k\right)\\\nonumber
&& =\tilde c_N u_0 R^{N-2}\left (\frac{1}{(4R-1)^{N-2}}- \frac{1}{(4R)^{N-2}}\right),
\end{eqnarray}
where in the last line we used  \eqref{eq:mad12-2_1}.
Finally by Lagrange theorem ve have 
\[u(x_0',1)\geq \tilde c_N\frac{(N-2) u_0}{4^{N-1}R}.\]
Therefore,  since  $u\in L^{\infty}(
\Sigma_{(0,2)})$,    we deduce
\[u(x_0)=u_0\leq C R,\]
for some constant $C=C(\gamma,\bar\lambda, \theta, N)$ that does not depend  on $R$. Since $x_0$ is arbitrary  we obtain that 
\[u(x)\leq C R\quad \text{in} \,\, \{x\in \mathbb R^N_+\,:\,  x_N > 2\}.\]
Scaling back, using \eqref{eq:ciserveallafine} and \eqref{eq:xnscaling} we obtain the thesis for $N\geq 3$.
The case $N=2$ follows repeating the same argument but replacing the fundamental solutions with the logarithmic one.
\end{proof}

\noindent It is straightforward to deduce the following 
\begin{cor}
Under the assumption {\bf (hp)}, $u$ has linear growth, namely there exits $c_1,c_2>0$ depending on $\gamma, \theta, N$ such that 
\[u(x)\leq c_1+c_2x_N.\]
\end{cor}
\begin{prop}\label{thm:gradbound} Under the assumption {\bf (hp)}, there exists $C=C(\gamma,  \theta,  N)$ such that the following hold
\begin{eqnarray}\nonumber
&(i)&\,\, |\nabla u|\leq C x_N^{\frac{1-\gamma}{\gamma +1}}\quad \text{in} \,\, \Sigma_{\bar\lambda},
\\\nonumber and \\\nonumber
&(ii)&|\nabla u|\leq C\quad \text{in} \,\, \mathbb R^N_+\setminus\Sigma_{\bar\lambda}.
\end{eqnarray}
\end{prop}
\begin{proof}
	Let us start noticing that, without loss of generality, we may and do assume that the solution is bounded in the strip $\Sigma_{2\bar\lambda}$.
 Let  now $P \in \Sigma_{\bar\lambda}$,  with $P=(x',x_N)$. Set $R =x_N$ and let us define 
\[u_R(x) =R^{-\frac{2}{\gamma+1}}{u(R x)}\quad \text{in}\,\,B_{\frac 12}\left(\frac{P}{R}\right).\]
Then $u_R$ satisfies 
\[-\Delta u_R =\frac{1}{u_R^\gamma}\quad \text{in}\,\,B_{\frac 12}\left(\frac{P}{R}\right).\] Exploiting Lemma  \ref{thm:Lazer}, we deduce that 
\[\frac{1}{u_R^\gamma}=\left(\frac{R^{\frac{2}{\gamma+1}}}{u(R x)}\right)^{\gamma}\leq \frac{4^{\frac{\gamma}{\gamma+1}}}{C^\gamma}\left(\frac{R^{\frac{2}{\gamma+1}}}{R^{\frac{2}{\gamma+1}}}\right)^{\gamma}:= C,\]
with $C=C(\gamma)$, i.e. $1/u_\delta^\gamma\in L^{\infty}(B_{1/2} \left({P}/{R}\right)).$ On the other hand Lemma \ref{thm:spr} we also get 
\[u_R(x)=R^{-\frac{2}{\gamma+1}}{u(R x)}\leq C,\]
where $C=C(\gamma, \bar\lambda)$.
By regularity estimates, see e.g. \cite[Theorem 3.9]{GT}
\[|\nabla u_R (x) |\leq C(\gamma, N) \quad \text{in}\,\,B_{\frac 14} \left(\frac{P}{R}\right).\] Consequently  we deduce
\[|\nabla u(R x) |\leq C R^{\frac{1-\gamma}{\gamma +1}}\quad \text{in}\,\,B_{\frac14} \left(\frac{P}{R}\right)\]
and  hence
\[|\nabla u(x) |\leq C R^{\frac{1-\gamma}{\gamma +1}}\quad \text{in}\,\,B_\frac{R}{4} (P),\]
thus proving $(i)$.
Arguing now in the same way, let us define 
\[u_R(x)=\frac{u(Rx)}{R} \quad \text{in}\,\,B_{\frac 12}\left(\frac{P}{R}\right).\] 
By Proposition \ref{thm:lineare} we have that $u_R\leq C(\gamma,\theta, N)$ and it satisfies 
\[-\Delta u_R=\frac{R^2}{R^{\gamma+1}}\frac{1}{u_R^\gamma}:=h  \quad \text{in}\,\,B_{\frac 12}\left(\frac{P}{R}\right),\]
where $h(x)\leq C(\gamma)$ in $B_{1/2}\left({P}/{R}\right)$ (see Lemma \ref{thm:Lazer}). By regularity estimates $|\nabla u_R|\leq C$ in $B_{1/4}\left({P}/{R}\right)$ and therefore $|\nabla u(x)|\leq C$ in $B_{R/4}\left({P}\right)$.
\end{proof}
We are now ready to prove the $1-D$ symmetry result.
\begin{thm}\label{thm1D}    Let u be a solution to \eqref{MP}. Under the assumption ${\bf (hp)}$ 
\[u(x)=u(x',x_N)=u(x_N)\quad \text{in}\,\,\mathbb R^N_+.\]
\end{thm}
\begin{proof}
Let $\tau,\sigma \in \mathbb R$, with $\sigma>0$, chosen opportunely later. Define $$u_{\tau,\sigma}=u(x+\tau e_i+\sigma e_N),$$ and  for $i=1,\ldots, N-1$. Obviously $-\Delta u_{\tau,\sigma} =1/u^\gamma_{\tau,\sigma}$ in $\mathbb R^N_+$. Setting $z:=u-u_{\tau,\sigma}$, we get
\begin{equation}\label{eq:raf}
-\Delta z=\frac{1}{u^\gamma}-\frac{1}{u_{\tau,\sigma}^\gamma}\quad\text{in}\,\,\mathbb R^N_+.
\end{equation}
In the following we  use \cite[Lemma 2.1]{BeCaNi}. From Lemma \ref{thm:spr} and Lemma  \ref{thm:Lazer} we infer that there exist constants $C_1, C_2$ such that
\begin{equation}\label{eq:doppstim}
C_1 x_N^{\frac{2}{\gamma+1}}\leq u(x)\leq C_2 x_N^{\frac{2}{\gamma+1}}\quad\text{in}\,\,\Sigma_{\bar\lambda}.
\end{equation} 
For $\sigma >0$, by \eqref{eq:doppstim} there exists $\rho>0$ and $\hat \lambda<\bar \lambda$ (actually think to $\hat \lambda\approx 0$) such that $u<\rho$ in $\Sigma_{\hat \lambda}$ and  $u_{\tau,\sigma}>2\rho$ in $\Sigma_{\hat \lambda}$, for all $\tau \in \mathbb R$. Defining the strip $D:=\mathbb R^N_+\setminus \Sigma_{\hat\lambda}$,  $z\leq 0$ on $\partial D$ holds. Moreover using Lagrange theorem jointly  to Proposition \ref{thm:gradbound}, we also get that  $z$ is bounded in $\overline D$.
 
\noindent Setting
\[c(x):= \left(\frac{1}{u^\gamma}-\frac{1}{u_{\tau,\sigma}^\gamma}\right)\frac{1}{u-u_{\tau,\sigma}},\]
we observe that $c(x)$ is continuous in $\overline D$ (indeed $u,u_{\tau,\sigma}\geq c>0$ in $D$, see \eqref{eq:doppstim}) and $c(x)\leq 0$ (in $D$) by its own definition. By \eqref{eq:raf} applying \cite[Lemma 2.1]{BeCaNi} to the problem
\begin{equation*}
\begin{cases}
\Delta z+c(x) z\geq 0 & \text{in}\,\, D\\
z\leq 0 & \text{on}\,\, \partial D,
\end{cases}
\end{equation*}
we obtain $z:=u-u_{\tau,\sigma}\leq 0$ in $D.$ We point out that already in $\Sigma_{\hat \lambda}\cup \{x_n=\hat \lambda\}$,  we have   $u-u_{\tau,\sigma}\leq 0$. 
Hence $u\leq u_{\tau,\sigma}$ in  $\mathbb R^N_+$. \ \\
Letting $\sigma \rightarrow 0$ we obtain 
 \[u\leq u_{\tau}\quad \text{for all}\,\,  \tau \in \mathbb R.\]
By the arbitrariness of $\tau$ we deduce that $u=u(x_N)$.
\end{proof}

\section{ODE Analysis}\label{sezode}
We start with the study of the one dimensional problem. We consider the following
\begin{eqnarray}\label{ODE}
	\begin{cases}
		\displaystyle - u''=\frac{1}{u^\gamma} & \text{in}\,\,\mathbb{R}_+\\
		u(t)>0 &  \text{in}\,\,\mathbb{R}_+\\
		u(0)=0.
	\end{cases}
\end{eqnarray}
It is straighforward to verify that the function
\begin{equation}\label{eq:esam3}
	u(t)=C_\gamma \ t^{\frac{2}{\gamma+1}}
\end{equation}
where 
\begin{equation}\label{eq:pen2}
	C_\gamma=\frac{(\gamma+1)^\frac{2}{\gamma+1}}{(2\gamma-2)^\frac{1}{\gamma+1}},
\end{equation}
is a solution of (\ref{ODE}).

\

\noindent {\bf A scaling argument.} Let $v\in C^2(\mathbb{R}_+)\cap C(\overline{\mathbb{R}_+})$ be a solution of problem (\ref{ODE}). Let 
\begin{equation}\label{rem:scaling}
	\sigma(t)=v_{\alpha,\lambda}(t):=\lambda^\alpha v(\lambda t),\end{equation} for a given $\lambda>0$ and $\alpha\in\mathbb{R}$. Then $\sigma(0)=0$, $\sigma(t)>0$ and for $t>0$
$$
\sigma''(t)=-\frac{1}{\sigma(t)^\gamma}\lambda^ {\alpha(1+\gamma)+2}.
$$
Choosing $\alpha=-{2}/{(1+\gamma)}$, then $\sigma$ satisfies (\ref{ODE}) too.  A similar computation showed that the same scaling works in the main problem \eqref{MP}.

\

Let us define, by means of \eqref{eq:esam3}, the function
$$
w(t):=C_\gamma\left(t+t^{\frac{2}{\gamma+1}}\right)=C_\gamma t+u(t)
$$
and notice that, since $u(t)<w(t)$ for $t>0$, 
\[\displaystyle w''(t)=u''(t)=-u(t)^{-\gamma}<-w(t)^{-\gamma}.\] Since $w(0)=u(0)=0$, $w(t)>0$ and
$$
-w''(t)\geq \frac{1}{w(t)^\gamma}\quad t>0,
$$
then $w$ is a supersolution for problem (\ref{ODE}). Moreover $w'(t)\to C_\gamma$ as $t\to +\infty$.

Taking into account the supersolution $w$, let us fix $t_0>0$ and consider the following problem
\begin{eqnarray}\label{ODE2}
	\begin{cases}
		\displaystyle - v''=\frac{1}{v^\gamma} & t>t_0\\
		v(t_0)>w(t_0)\\
		v'(t_0)>w'(t_0).
	\end{cases}
\end{eqnarray}
\begin{prop}\label{pro:3}
	Each solution of problem (\ref{ODE2}) is such that $v(t)>w(t)$ for $t\geq t_0$ and there exists (finite) $\displaystyle \lim_{t\to \infty}v'(t)\geq C_\gamma$.
\end{prop}
\begin{proof}
	A unique local solution for problem (\ref{ODE2}) there exists; indeed, it can be proved, that   the solution is defined in the whole $[t_0, +\infty)$ since it is concave. Moreover
	$$
	\left(v'(t_0)-w'(t_0)\right)'\geq \frac{1}{w(t_0)^\gamma}-\frac{1}{v(t_0)^\gamma}.
	$$ 
	Since $\left(v'(t_0)-w'(t_0)\right)'>0$,  there exists $\delta>0$ such that for all $t\in [t_0,t_0+\delta)$,  $\left(v'(t)-w'(t)\right)'>0$.
	Actually  $\left(v'(t)-w'(t)\right)'>0$ for each $t>t_0$; if not, denoting by $\tau:=\sup \{t>t_0:\left(v'(t)-w'(t)\right)'>0\}$, it follows that
	\begin{equation*}
		0= (v'(\tau)-w'(\tau))'\geq \frac{1}{w(\tau)^\gamma}-\frac{1}{v(\tau)^\gamma},
	\end{equation*}
	hence
	\begin{equation}\label{contr}
		\frac{1}{v(\tau)^\gamma}\geq\frac{1}{w(\tau)^\gamma}.
	\end{equation}
	Since $(v'-w')$ is continuous and (strictly) increasing on the interval $[t_0,\tau)$, and $v(t_0)>w(t_0)$,  therefore $v(\tau)>w(\tau)$.  This contradict (\ref{contr}).
	As a consequence, $v(t)>w(t)$ for $t\geq t_0$.
	
	The solution of (\ref{ODE2}) is positive on $[t_0,+\infty)$, therefore $v''(t)$ is negative on the same interval. This implies that $v'(t)$ is decreasing and its limit there exists for $t\to\infty$.
	Thus $\displaystyle \lim_{t\to \infty}v'(t)\geq \lim_{t\to \infty}w'(t)= C_\gamma$.
\end{proof}
\begin{lem}
	For any $L\in \mathbb{R}_+$, there exists a solution $\tilde{v}$ for the problem
	\begin{eqnarray}\label{ODE4}
		\begin{cases}
			\displaystyle - v''=\frac{1}{v^\gamma} & t>0\\
			v(t)>0 &t>0 \\
			v(0)=0 \displaystyle \lim_{t\to+\infty}v'(t)=L.
		\end{cases}
	\end{eqnarray}
\end{lem}
\begin{proof}
	We start proving that, 
	choosing  $ v'(t_0)>{v(t_0)}/{t_0}$ in \eqref{ODE2}, then   there exists $\tau_0\in (0,t_0]$ such that $v(t)$ given in Proposition \ref{pro:3},  can be extended as a solution of
	\begin{eqnarray*}
		\begin{cases}
			\displaystyle - v''=\frac{1}{v^\gamma} & t>\tau_0\\
			v>0 & t>\tau_0\\
			v(\tau_0)=0.
		\end{cases}
	\end{eqnarray*}
	Indeed each extension of $v(t)$ for $t<t_0$ is such that $v''(t)\leq 0$ and therefore  the graph of $v(t)$ lies below to the tangent line to $v(t)$ in $(t_0,v(t_0))$.  Since $ v'(t_0)>{v(t_0)}/{t_0}$, then  a such  $\tau_0>0$ exists.
	
	Let $v_0(t)$ be a such   solution,   let us define $\tilde{v}(t):=v_0(t+\tau_0)$. Then $\tilde{v}(0)=0$ and verifies  (\ref{ODE4}). 
\end{proof}

\begin{thm}\label{Thmunic}
	Let $M>0$ be fixed. Then there exists  a solution to
	\begin{eqnarray}\label{ODElim}
		\begin{cases}
			\displaystyle - w''=\frac{1}{w^\gamma} & t>0\\
			w(t)>0 &t>0 \\
			w(0)=0 \displaystyle \lim_{t\to+\infty}w'(t)=M,
		\end{cases}
	\end{eqnarray}
	and the solution is unique.
\end{thm}
\begin{proof}
	Let $v$ be a solution of problem (\ref{ODE4}).  Let 
	$$\displaystyle \lambda:=\left(\frac{M}{L}\right)^\frac{\gamma+1}{\gamma-1},$$ where $\displaystyle L:= \lim_{t\to+\infty}v'(t)$. By the scaling 
	\eqref{rem:scaling}, we have
	$$
	w(t)=\lambda^{-\frac{2}{\gamma+1}}v(\lambda t)=\left(\frac{M}{L}\right)^{-\frac{2}{\gamma-1}}v\left(\left(\frac{M}{L}\right)^\frac{\gamma+1}{\gamma-1}t\right)
	$$
	is a solution of (\ref{ODElim}) and since $v'(t)\to L$ as $t\to+\infty$, $w'(t)\to M$
	as $t\to +\infty$.\ \\
	
	About the uniqueness, let us consider (by contradiction) $w_1,w_2$ two different solutions of (\ref{ODElim}). At first, let us assume that there exists $t_0>0$, the smallest value for which $w_1(t_0)=w_2(t_0)$. Taking into account the initial condition $w_1(0)=w_2(0)=0$ and that  $w_1,w_2$ are continuous, by the weak comparison principle it follows that $w_1(t)=w_2(t)$ on the interval $[0,t_0]$. Indeed, let us suppose without loss of generality, that $w_1\leq w_2$ in $[0,t_0]$; for any $\varepsilon>0$, let $\varphi:=(w_2-w_1-\varepsilon)$ be a test function for problems \eqref{ODE4} and \eqref{ODElim}. So we have
	\begin{eqnarray*}
		\int_0^{t_0}|\nabla (w_2-w_1-\varepsilon)|^2dx=\int_0^{t_0}\left( \frac{1}{w_2^\gamma}-\frac{1}{w_1^\gamma}\right)(w_2-w_1-\varepsilon)\,dx\leq 0.
	\end{eqnarray*}
	Then $w_1 = w_2+\varepsilon$ in $[0,t_0]$ for all $\varepsilon>0$, therefore $w_1=w_2$ in $[0,t_0]$. As a rule $w_1(t_0)=w_2(t_0)$ and $w_1'(t_0)=w_2'(t_0)$ then $w_1= w_2$ in $\mathbb{R}^+$ by uniqueness for ODEs (note that $w_1,w_2>0$ in $\mathbb{R}^+$ so that $-w''=w^{-\gamma}$ is a regular ODE). Consequently, different solutions $w_1$ and $w_2$ do not cross. 
	
	\
	
	From now on we may assume that $w_1 < w_2$ for all $t\in\mathbb{R}^+$. Notice that,
	$$
	(w_1'-w_2')'=\frac{1}{w_2^\gamma}-\frac{1}{w_1^\gamma}<0\quad \mbox{ in }\,\, \mathbb{R}^+.
	$$
	Since $w_1'(t),w_2'(t)\to M$ as $t\to+\infty$, then $w_1'(t)-w_2'(t)>0$ for all $t\in\mathbb{R}^+$ namely $w_1-w_2$ should be increasing in $\mathbb{R}^+$ causing $w_1=w_2$ in $\mathbb{R}^+$.
\end{proof}

\section{Conclusiom: proof of Theorem \ref{mainthm}}\label{Sec4}
Once that Theorem \ref{thm1D} is in force and therefore we know that 
\[
u(x)=u(x_N),
\]
we get that $u$ is a positive solution to
\[
-u''=\frac{1}{u^\gamma}\quad \mbox{ in }\,\, \mathbb{R}^+,
\]
\noindent with $u(0)=0$. Therefore the ODEs analysis of Section \ref{sezode} allows us to conclude that, either the solution is given by  \eqref{eq:esam3} or has linear growth and is completely classified by Theorem \ref{Thmunic}, taking into account   the scaling in \eqref{rem:scaling}.

\vspace{1cm}

\begin{center}{\bf Acknowledgements}\end{center}  
L. Montoro and B. Sciunzi are partially supported by PRIN project 2017JPCAPN (Italy): Qualitative and quantitative aspects of nonlinear PDEs, and L. Montoro by  Agencia Estatal de Investigaci\'on (Spain), project PDI2019-110712GB-100.

\

\begin{center}
{\sc Data availability statement}

\

All data generated or analyzed during this study are included in this published article.
\end{center}

\

\begin{center}
	{\sc Conflict of interest statement}
	
	\
	
	The authors declare that they have no competing interest.
\end{center}

\bibliographystyle{elsarticle-harv}

\end{document}